\documentclass{article}

\usepackage{microtype}
\usepackage{graphicx}
\usepackage{subfigure}
\usepackage{booktabs} 
\usepackage{algorithm}
\usepackage{algorithmic}
\usepackage{hyperref}

\usepackage[final]{sods_2023}

\usepackage{amsmath}
\usepackage{amssymb}
\usepackage{mathtools}
\usepackage{amsthm}
\usepackage{multicol}
\usepackage[mathscr]{euscript}

\usepackage[capitalize,noabbrev]{cleveref}

\def\av{\mathbf{a}}
\def\uv{\mathbf{u}}
\def\vv{\mathbf{v}}

\def\am{\mathbf{A}}
\def\bm{\mathbf{B}}
\def\cm{\mathbf{C}}
\def\qm{\mathbf{Q}}
\def\xm{\mathbf{X}}
\def\zm{\mathbf{Z}}
\def\tm{\mathbf{T}}
\def\dm{\mathbf{D}}
\def\ym{\mathbf{Y}}
\def\um{\mathbf{U}}
\def\vm{\mathbf{V}}
\def\wm{\mathbf{W}}

\def\prox{\mathrm{prox}}
\def\rk{\mathrm{rank}}
\def\dg{\mathrm{diag}}
\def\sp{\mathrm{supp}}
\def\xt{\mathscr{X}}
\def\zt{\mathscr{Z}}

\DeclareMathOperator*{\argmin}{argmin}

\theoremstyle{plain}
\newtheorem{theorem}{Theorem}[section]

\newtheorem{lemma}[theorem]{Lemma}

\theoremstyle{definition}
\newtheorem{definition}[theorem]{Definition}

\theoremstyle{remark}

\usepackage[textsize=tiny]{todonotes}

\title{Strictly Low Rank Constraint Optimization \\ --- An Asymptotically $\mathcal{O}(\frac{1}{t^2})$ Method}
\author{
	Mengyuan Zhang, Kai Liu\\
	\texttt{\{mengyuz, kail\}@clemson.edu}\\
	Clemson University\\
}

\begin{document}

\maketitle

\begin{abstract}
We study a class of non-convex and non-smooth problems with \textit{rank} regularization to promote sparsity in optimal solution. We propose to apply the proximal gradient descent method to solve the problem and accelerate the process with a novel support set projection operation  on the singular values of the intermediate update. We show that our algorithms are able to achieve a convergence rate of $O(\frac{1}{t^2})$, which is exactly same as Nesterov's optimal convergence rate for first-order methods on smooth and convex problems. Strict sparsity can be expected and the support set of singular values during each update is monotonically shrinking, which to our best knowledge, is novel in momentum-based algorithms.

\end{abstract}

\section{Introduction}\label{sec:intro}
Sparsity-induced optimization has achieved great success in many data analyses, and low-rank regularization is a powerful tool to impose sparsity. 
In this paper, we study the class of non-convex and non-smooth sparse learning problems in discrete space presented as follows:
\begin{equation}\label{eq:fx}
    \min F(\xm) = g(\xm) + h(\xm),
\end{equation}
where $\xm \in \mathbb{R}^{n \times k}, h(\xm) = \lambda\cdot rank (\xm) \ (\lambda > 0)$.
Countless problems in machine learning~\cite{shang2017nonnegative, su2020low}, computer vision~\cite{haeffele2014structured, he2015total}, and signal processing~\cite{dong2014compressive, zhou2014regularized} naturally fall into this template, including but not limited to rank regularized matrix factorization for recommendation, image restoration and clustering, compressive sensing, and multi-task learning. 

Besides the efforts trying to relax the $rank$ regularization with the nuclear norm, which is the tightest convex envelop, a class of non-convex optimization
algorithms has been developed to solve rank-regularized problems with vanilla objectives. The benefit is obvious that it will bring strict sparse solution instead of low energy of singular values. For example, the cardinality constraint is studied for $M$-estimation problems by Iterative Hard-Thresholding (IHT) method~\cite{blumensath2008iterative}. In addition, the alternating direction method of multipliers is also explored to apply on the rank-regularized problem~\cite{qu2023efficient}. Moreover, general iterative shrinkage and thresholding algorithm has been proposed to solve non-convex sparse regularization problems~\cite{gong2013multiple}. We refer the readers to the papers aforementioned and the references therein. In this paper, we use the proximal gradient descent method to obtain a sparse sub-optimal solution for Eq.~(\ref{eq:fx}). Also, with a \textit{support set projection} operation on the singular values of the matrix $\xm$, our proposed algorithm can be accelerated with a faster convergence rate. We explicitly list our contributions as following: 
\begin{itemize}
	\item We show that with the proximal gradient descent method, the sequence obtained converges to a critical point with \textit{zero} gradient, at a convergence rate of $F(\xm^{t+1}) - F(\xm^*) \leq O(\frac{1}{t})$.
	\item We propose two accelerated proximal gradient descent methods (monotone and non-monotone decreasing) which can \textit{asymptotically} achieve Nesterov Accelerated Gradient (NAG) convergence rate ($O(\frac{1}{t^2})$) which originally is supposed for convex problems~\cite{nesterov2003introductory} and we give the rigorous proof.
	\item The proposed algorithms can indeed admit strict sparsity, as the support set of singular values in each update is a subset of the initialized set and keeps shrinking during the update. 
\end{itemize}

\section{Algorithms and Convergence Analysis}\label{sec:alg}
 Throughout this paper, we use uppercase bold letters for matrices, lowercase bold letters for vectors, and lowercase letters for scalars. We use superscript to represent the current iteration. $\sigma_i (\cdot)$ is  the $i$-th largest singular value of a matrix. $\sigma_{min} (\cdot)$ and $\sigma_{max} (\cdot)$ represent the smallest and the largest singular value of a matrix, respectively. $\sigma(\cdot)$ indicates the vector formed by the singular values of a matrix, and $\dg(\sigma(\cdot))$ denotes the diagonal matrix formed by the vector. $|\cdot|$ denotes the cardinality of a set, and supp$(\cdot)$ denotes the support of a vector. For a matrix $\xm \in \mathbb{R}^{n \times k}$, let $\sigma(\xm) = (\sigma_{max} (\xm), \sigma_2 (\xm), \dots, \sigma_{min} (\xm))$, then we can see rank$(\xm) = \|\sigma(\xm)\|_0$. 
 $S = \sp(\sigma(\xm^0))$, $\xm^0$ is the initialization. 
 The proximal mapping associated with $h(\cdot)$ is defined as $\textrm{prox}_h(\um) = \argmin_{\vm} h(\vm) + \frac{1}{2} \|\um-\vm\|^2_F$. 
We make some mild assumptions regarding $g(\xm)$:
1) g($\cdot)$ is convex, and $\nabla g(\cdot)$ has bounded singular value $\sigma_{max}(\nabla g(\xm))) \leq G$ for any $\xm$;
2) g($\cdot)$ has L-Lipschitz continuous gradient, $\| \nabla g(\xm_1) - \nabla g (\xm_2) \|_F \leq L\|\xm_1 - \xm_2\|_F$.
Regarding $g(\cdot)$, we set the squared loss $\|\ym - \dm \xm\|_F^2$ to demonstrate the algorithm for simplicity, where $\ym \in \mathbb{R}^{d \times k}, \dm \in \mathbb{R}^{d \times n},\xm \in \mathbb{R}^{n \times k}$. It is worth noting that the objective can be extended to other loss functions.
\subsection{Proximal Gradient Descent for Low-Rank Approximation}\label{subsec:PGD}
In the $t$-th iteration of the proximal gradient descent (PGD) method, gradient descent is applied on the squared loss function $g(\xm)$ to get
$\xm^t - s\nabla g(\xm^t)$, where $s \geq 0$ is the step size in gradient descent and $\frac{1}{s}$ is typically larger than the Lipschitz continuous constant $L$ of $g(\xm)$. After applying the proximal mapping on $\xm^t - s\nabla g(\xm^t)$ we can get $\xm^{t+1}$:
\begin{equation}\begin{split}\label{eq:proximal}
    \xm^{t+1} =  \prox_h (\xm^t - s\nabla g(\xm^t))  
    =  T_{\sqrt{2\lambda s}} (\xm^t - s\nabla g(\xm^t)).
\end{split}\end{equation}
$T_{\theta} (\cdot)$ is the singular value hard thresholding operator defined as 
$T_{\theta}(\qm) = \um \Sigma_{\theta} \vm^T$, where
$\qm = \um \Sigma \vm^T$ is the singular value decomposition and 
$ \Sigma_{\theta} (i,i) = \begin{cases}
      0,  &|\Sigma(i,i)| \leq \theta, \\
      \Sigma(i,i),  &\textrm{otherwise.} \\
     \end{cases}$

The optimization algorithm to minimize problem Eq.~(\ref{eq:fx}) by PGD is summarized in Algorithm~\ref{alg:1}.
\begin{algorithm}[h]
	\caption{Proximal Gradient Descent for the Rank-Regularized Problem}
	\label{alg:1}
	\begin{algorithmic}
		\STATE {\bfseries Input:} Initialization $\xm^0$, step size $s$, regularization parameter $\lambda$.
		\FOR{$t=0, \dots $} 
		\STATE Update $\xm^{t+1}$ according to Eq.~(\ref{eq:proximal})
		\ENDFOR
	\end{algorithmic}
\end{algorithm}
%
%

We present the analysis of the convergence rate of Algorithm~\ref{alg:1}. Before that, we first show the support set of the singular value vector shrinks, and then the rank of obtained solutions decreases, also the objective has a sufficient decrease during each update.
\begin{lemma}\label{lem:1}
If $s \leq \min \{\frac{2\lambda}{G^2}, \frac{1}{L}\}$, then 
$\sp (\sigma(\xm^{t+1})) \subseteq \sp (\sigma(\xm^{t})), \textrm{for } t \geq 0$,
and
$\rk(\xm^{t+1}) \leq  \rk(\xm^{t}), \textrm{for } t \geq 0$,
which means the support of the singular value vectors of the sequence $\{\xm^t\}_t$ shrinks, also the rank of the sequence $\{\xm^t\}_t$ decreases.
\end{lemma}

\begin{lemma}\label{lem:2}
With $s \leq \min \{\frac{2\lambda}{G^2}, \frac{1}{L}\}$, the sequence of the objective $\{F(\xm^t)\}_t$ is monotonically non-increasing, and the following inequality holds for all $t \geq 0$:
\begin{equation}
    F(\xm^{t+1}) \leq F(\xm^t) - (\frac{1}{2s} - \frac{L}{2})\|\xm^{t+1} - \xm^t\|^2_F.
\end{equation}
\end{lemma}
To eliminate the concern that very small step size is required to ensure Lemma~\ref{lem:1} with the prerequisite $s \leq \min \{\frac{2\lambda}{G^2}, \frac{1}{L}\}$, we show that $\frac{1}{L}$ is no larger than $\frac{2\lambda}{G^2}$ with high probability, thus the choice of step size $s$ in Lemma~\ref{lem:1} would not be smaller than $ \frac{1}{L}$ with high probability. 
\begin{theorem}\label{thm:1}
\cite{yang2020fast} \ Suppose $\dm \in \mathbb{R}^{d \times n} (n \geq d)$ is a random matrix with elements i.i.d. sampled from the standard Gaussian distribution \textit{N}(0,1), then 
\begin{equation}
    \textrm{Probability}(\frac{1}{L} \leq \frac{2\lambda}{G^2}) \geq 1-e^{-\frac{a^2}{2}} - ne^{-a},
\end{equation}
if 
$n \geq (\sqrt{d} + a + \sqrt{\frac{(d+2\sqrt{da}+2a)(x_0 + \lambda |S|))}{\lambda}})^2$,
where $x_0 = \|\ym - \dm \xm^0\|^2_F$, $S = \sp(\sigma(\xm^0))$, and $a$ can be chosen as $a_0 \textrm{log}n$ for $a_0 > 0$.
\end{theorem}

The sequence $\{\xm^t\}_t$ can be segmented into the following $|S|+1$ subsequences $\{\xt^{k'}\}_{k'=0}^{|S|}$ with the definition as follows:
\begin{equation}
    \xt^{k'} = \{\xm^t: |supp(\sigma(\xm^t))| = k', t \geq 0\}, 0 \leq k' \leq |S|. 
\end{equation}

With the definition defined in Definition~\ref{def:1}, the nonempty subsequences in $\{\xt^{k'}\}_{k'=0}^{|S|}$ form a disjoint cover of $\{\xt^{t}\}_t$ and they are in descending order of rank.

\begin{definition}\label{def:1}
Subsequences with shrinking support:
All the $K \leq |S|+1$ nonempty subsequences among $\{\xt^{k'}\}_{k'=0}^{|S|}$ are defined to be subsequences with shrinking support, denoted by $\{\xt^{k}\}_{k=1}^{K}$. The subsequences with shrinking support are ordered with decreasing support size of singular value vectors, i.e. $|\sp(\sigma(\xm^{t_2}))| < |\sp(\sigma(\xm^{t_1}))|$ for any $\xm^{t_1} \in \xt^{k_1}$ and $\xm^{t_2} \in \xt^{k_2}$ with any $1\leq k_1 < k_2 \leq K$.
\end{definition}

Based on the above definition, we have the following lemma about the property of subsequences with shrinking support:
\begin{lemma}\label{lem:3}
(a) All the elements of each subsequence $\xt^k$ ($k = 1, \dots, K)$ in the subsequences with shrinking support have the same support. In addition, for any $1 \leq k_1 < k_2 \leq K$ and any $\xm^{t_1} \in \xt^{k_1}$, and $\xm^{t_2} \in \xt^{k_2}$, we have $t_1 < t_2$ and $\sp(\sigma(\xm^{t_2})) \subset \sp(\sigma(\xm^{t_1}))$. 

(b) All the subsequences except for the last one, $\xt^k$ ($k = 1, \dots, K-1)$ have finite size, and $\xt^K$ have an infinite number of elements, and there exists some $t_0 \geq 0$ such that $\{\xm^t\}_{t=t_0}^\infty \subseteq \xt^K$.
\end{lemma}

The following theorem shows the convergence property of the sequence $\{\xm^{t}\}_t$:
\begin{theorem}\label{thm:2}
    Suppose $s \leq \min \{\frac{2\lambda}{G^2}, \frac{1}{L}\}$. and $\xm^*$ is a limit point of $\{\xm^{t}\}_{t=0}^{\infty}$, 
    and $\sigma(\xm^*)$ is a limit point of $\{\sigma(\xm^{t})\}_{t=0}^{\infty}$, 
    then the sequence $\{\xm^{t}\}_{t=0}^{\infty}$ generated by Algorithm~\ref{alg:1} converges to $\xm^*$,
    $\xm^*$ is a critical point of F($\cdot$), and $\sp(\sigma(\xm^*)) = S^*$, where $S^*$ is the support of any element in $\xt^K$.
    Moreover, there exists $t_1 \geq 0$ such that for all $m \geq t_1$, we have
    \begin{equation}
        F(\xm^{m+1}) - F(\xm^*) \leq \frac{1}{2s(m-t_1+1)}\|\xm^{t_1} - \xm^*\|_F^2.
    \end{equation}
\end{theorem}

 \begin{multicols}{2}
	\begin{figure}[H]
		\begin{center}
			\centerline{\includegraphics[width=.97\columnwidth]{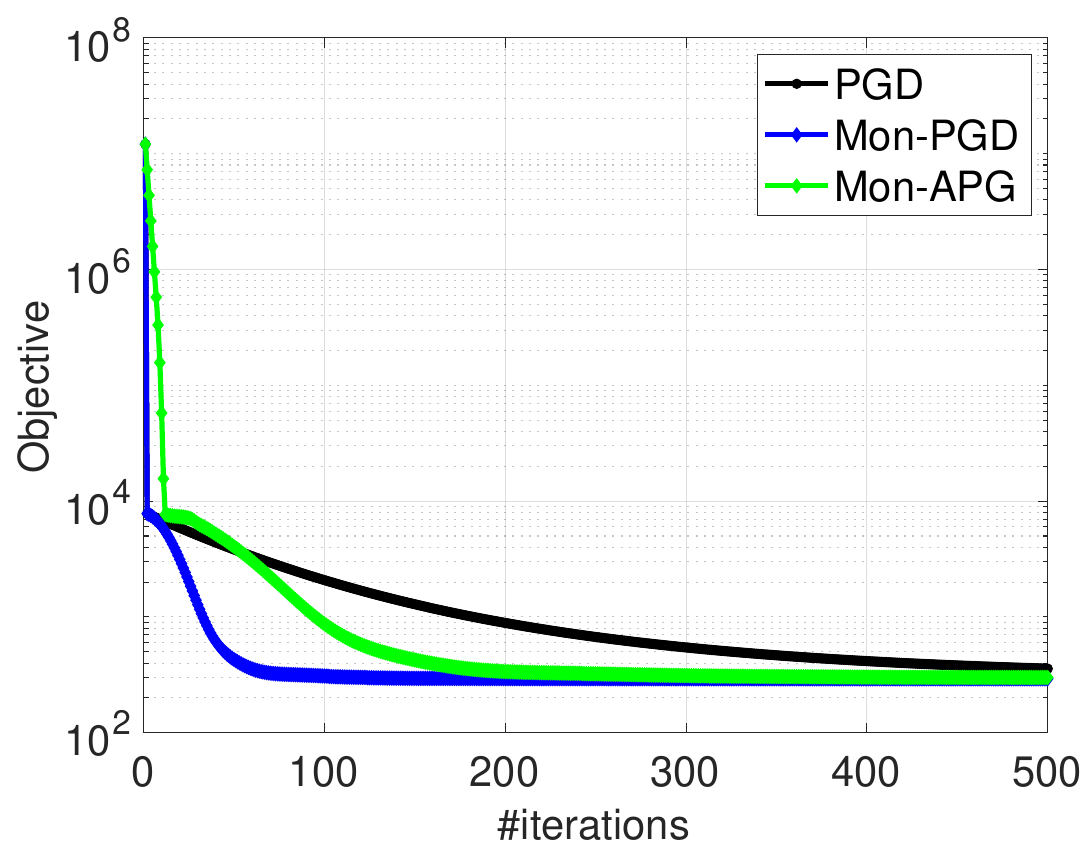}}
			\caption{The convergence for Algorithm~\ref{alg:1} PGD, Algorithm~\ref{alg:3} Mon-PGD, and Mon-APG.}
			\label{fig:1}
		\end{center}
	\end{figure}
	\begin{algorithm}[H]
		\caption{Non-monotone Accelerated Proximal Gradient Descent}
		\label{alg:2}
		\begin{algorithmic}
			\STATE {\bfseries Input:} Initialize $\xm^0$,
			$\xm^1 = \xm^0$, $\alpha^0 = 1, s, \lambda$.
			\FOR{$t=1, \dots $} 
			\STATE Update $\um^{t}, \vm^{t}, \xm^{t+1}, \alpha^{t+1}$ by Eq.~(\ref{eq:non1}).
			\ENDFOR
		\end{algorithmic}
	\end{algorithm}
	\begin{algorithm}[H]
		\caption{Monotone Accelerated Proximal Gradient Descent}
		\label{alg:3}
		\begin{algorithmic}
			\STATE {\bfseries Input:} Initialization 
			$\zm^1 = \xm^1 = \xm^0$, $\alpha^0, s, \lambda$.
			\FOR{$t=1, \dots $} 
			\STATE Update $\um^{t}, \vm^t, \zm^{t+1}, \alpha^{t+1}, \xm^{t+1}$ by (\ref{eq:mo1}).
			\ENDFOR
		\end{algorithmic}
	\end{algorithm}
\end{multicols}
\subsection{Non-monotone Accelerated Proximal Gradient Descent with Support Projection}\label{subsec:acc}
\begin{figure}[h!]
	\centering 
	\subfigure[Initialized $\xm^0$]{\label{fig:11}\includegraphics[width=42mm]{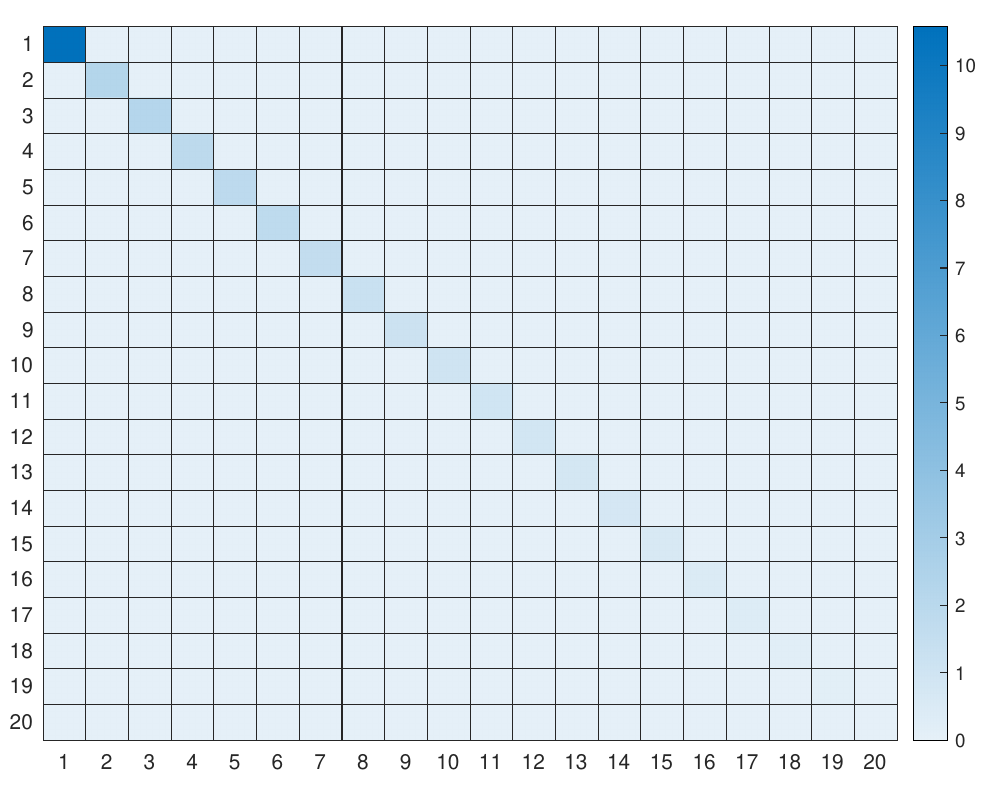}}
	\subfigure[$\xm^t$ during update]{\label{fig:12}\includegraphics[width=42mm]{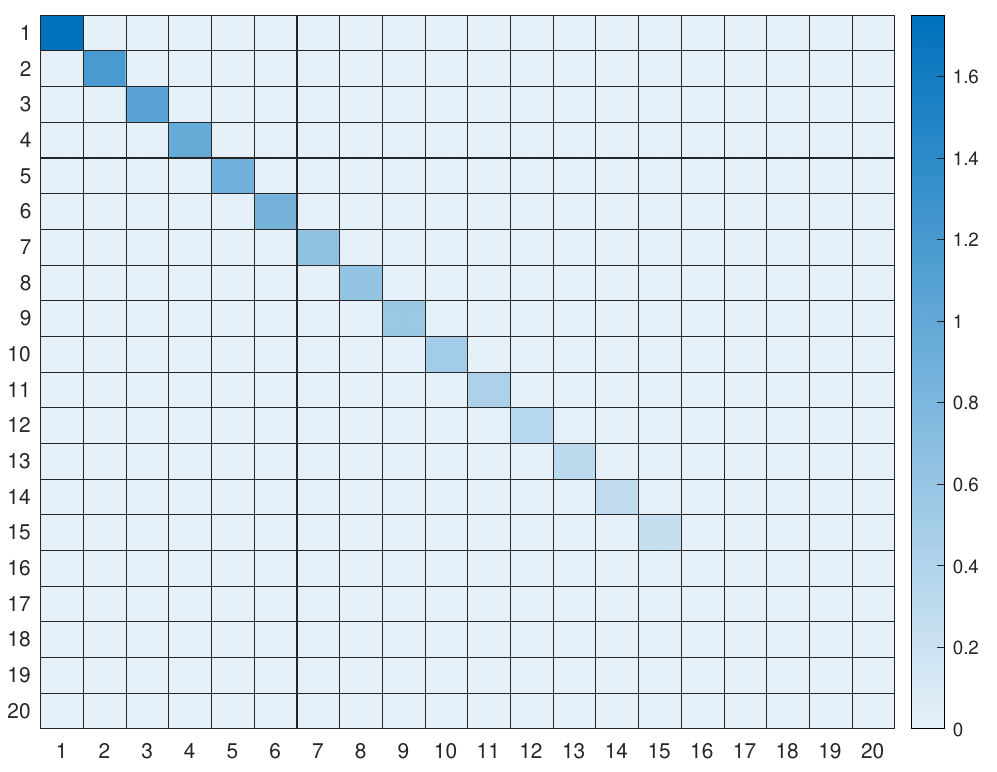}}
	\subfigure[Output $\xm^T$]
	{\label{fig:13}\includegraphics[width=42mm]{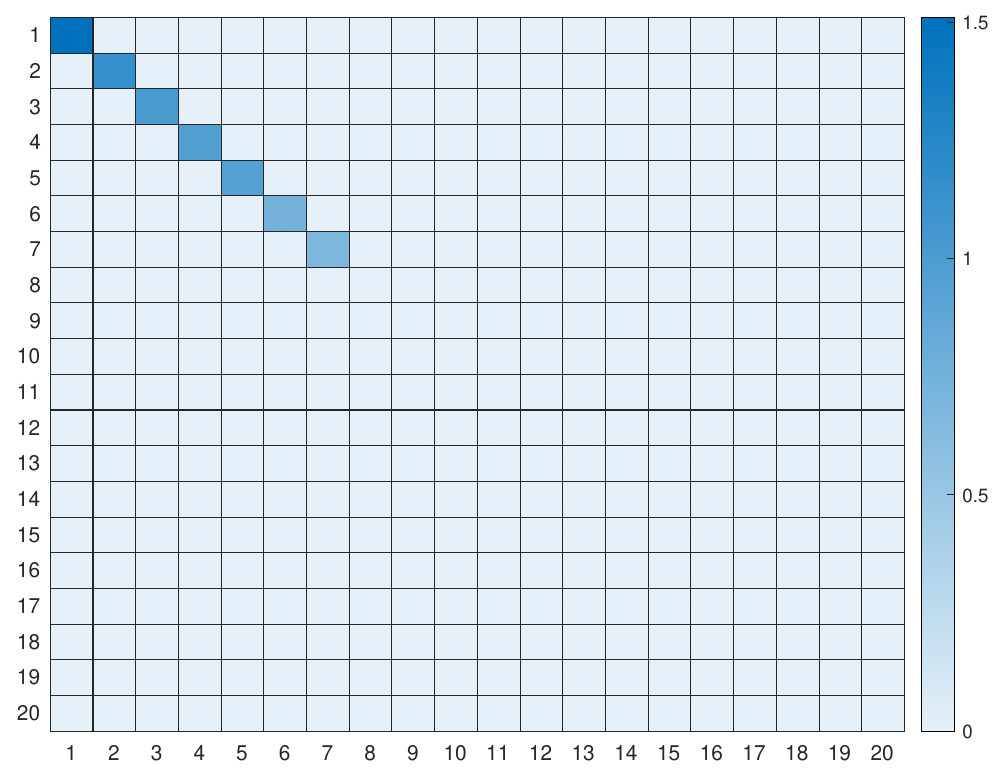}}
	\caption{Support Shrinkage of \textit{singular values} in Algorithm~\ref{alg:1} PGD.}
	\label{fig:2}
\end{figure}
In the non-monotone accelerated proximal gradient descent with support projection, the update process in the $t$-th iteration is as follows:
\begin{equation}\begin{split}\label{eq:non1}
    \um^t = \xm^t + \frac{\alpha^{t-1}-1}{\alpha^t}(\xm^t - \xm^{t-1}),
     \vm^t = P_{\sp(\sigma(\xm^t))} (\um^t), \\
    \xm^{t+1} = \prox_h (\vm^t - s\nabla g(\vm^t)) , 
     \alpha^{t+1} = \frac{\sqrt{1+4(\alpha^{t}) ^2}+1}{2},
\end{split}\end{equation}
where $P_{\sp(\sigma(\xm^t))}(\cdot)$ is the support projection operator which projects the singular value vector of a matrix to the support of the singular value vector of $\xm^t$ as $P_{\sp(\sigma(\xm^t))}(\tm) = \am \Sigma_{\textrm{projected}} \bm^T,$
with $\tm = \am \Sigma \bm^T$ being the singular value decomposition, and 
$\Sigma_{\textrm{projected}} (i,i) = \Sigma(i,i) $ if $i$ is in the support of $\sigma(\xm^t)$, otherwise $=0$. Here the support projection is employed to enforce the support shrinkage property mentioned in Lemma~\ref{lem:3}. The algorithm for the non-monotone accelerated proximal gradient descent with support projection is summarized in Algorithm~\ref{alg:2}. Lemma~\ref{lem:4} shows the support shrinkage property for Algorithm~\ref{alg:2}, and Theorem~\ref{thm:3} presents the convergence rate. 
\begin{lemma}\label{lem:4}
    The sequence $\{\xm^t\}_{t=1}^{\infty}$ generated by Algorithm~\ref{alg:2} satisfies $\sp(\sigma(\xm^{t+1})) \subseteq \sp(\sigma(\xm^{t})), t \geq 1$.
\end{lemma}
\begin{theorem}\label{thm:3}
      Suppose $s \leq \min \{\frac{2\lambda}{G^2}, \frac{1}{L}\}$. and $\xm^*$ is a limit point of $\{\xm^{t}\}_{t=0}^{\infty}$ generated by Algorithm~\ref{alg:2}, 
  then there exists $t_0 \geq 1$ such that for all $m \geq t_0$, we have
    \begin{equation}
        F(\xm^{m+1}) - F(\xm^*) \leq \frac{4}{(m+1)^2}[\frac{1}{2s} \|(\alpha^{t_0-1}-1)\xm^{t_0-1} - \alpha^{t_0-1}\xm^{t_0} +\xm^*\|_F^2 + (\alpha^{t_0-1})^2(F(\xm^{t_0})-F(\xm^*))].
    \end{equation}
\end{theorem}
\subsection{Monotone Accelerated Proximal Gradient Descent with Support Projection}\label{subsec:mon}
To ensure the objective is non-increasing, we introduce the following  algorithm which is summarized in Algorithm~\ref{alg:3}~\cite{beck2017first}:
\begin{equation}\begin{split}\label{eq:mo1}
    \um^t = \xm^t + \frac{\alpha^{t-1}-1}{\alpha^t}(\xm^t - \xm^{t-1}) 
    + \frac{\alpha^t-1}{\alpha^t}(\zm^t - \xm^t),
     \vm^t &= P_{\sp(\sigma(\zm^t))} (\um^t), \\
    \zm^{t+1} = \prox_h (\vm^t - s\nabla g(\vm^t)) , 
     \alpha^{t+1} = \frac{\sqrt{1+4(\alpha^{t})^ 2}+1}{2},
     \xm^{t+1} &= \begin{cases}
         \zm^{t+1} & \text{if } F(\zm^{t+1}) \leq F(\xm^t), \\
         \xm^t & \text{otherwise.} \\
     \end{cases}
\end{split}\end{equation}
Lemma~\ref{lem:5} shows the support shrinkage property and Theorem~\ref{thm:4} presents the convergence rate. 
\begin{lemma}\label{lem:5}
    The sequence $\{\zm^t\}_{t=1}^{\infty}$ and $\{\xm^t\}_{t=1}^{\infty}$ generated by Algorithm~\ref{alg:3} satisfies
$\sp(\sigma(\zm^{t+1})) \subseteq \sp(\sigma(\zm^{t})), t \geq 1,
        \sp(\sigma(\xm^{t+1})) \subseteq \sp(\sigma(\xm^{t})), t \geq 1$.
\end{lemma}
\begin{theorem}\label{thm:4}
      Suppose $s \leq \min \{\frac{2\lambda}{G^2}, \frac{1}{L}\}$. and $\xm^*$ is a limit point of $\{\xm^{t}\}_{t=0}^{\infty}$ generated by Algorithm~\ref{alg:3}, 
  then there exists $t_0 \geq 1$ such that for all $m \geq t_0$, we have
    \begin{equation}
        F(\xm^{m+1}) - F(\xm^*) \leq \frac{4}{(m+1)^2}[\frac{1}{2s} \|(\alpha^{t_0-1}-1)\xm^{t_0-1} - \alpha^{t_0-1}\zm^{t_0} +\xm^*\|_F^2  + (\alpha^{t_0-1})^2(F(\xm^{t_0})-F(\xm^*))].
    \end{equation}
\end{theorem}
We leave all the detailed proof in the Appendix due to space limit.

\bibliography{example_paper}
\bibliographystyle{icml2023}

\newpage
\appendix
\onecolumn
\section{Proofs for Subsection ~\ref{subsec:PGD}}

\begin{lemma}\label{lem:1appen}
If $s \leq \min \{\frac{2\lambda}{G^2}, \frac{1}{L}\}$, then 
\begin{equation}
    \sp (\sigma(\xm^{t+1})) \subseteq \sp (\sigma(\xm^{t})), \textrm{for } t \geq 0,
\end{equation}
and
\begin{equation}
   \rk(\xm^{t+1}) \leq  \rk(\xm^{t}), \textrm{for } t \geq 0,
\end{equation}
which means the support of the singular value vectors of the sequence $\{\xm^t\}_t$ shrinks, also the rank of the sequence $\{\xm^t\}_t$ decreases.
\end{lemma}

\begin{proof}
Let $\bar{\xm}^{t+1} = \xm^t - s\nabla g(\xm^t), \qm^t = -s\nabla g(\xm^t),$ thus we have $\bar{\xm}^{t+1} = \xm^t + \qm^t$. With Weyl's inequality 
\begin{equation}
    \sigma_{i+j-1} (\am + \bm) \leq \sigma_i(\am) + \sigma_j(\bm),
\end{equation}
we get 
\begin{equation}
    \sigma_i(\bar{\xm}^{t+1}) \leq \sigma_i(\xm^t) + \sigma_1(\qm^t) = \sigma_1(\qm^t) \leq sG,
    \textrm{for all } i \textrm{ where } \sigma_i(\xm^t) = 0.  
\end{equation}

With $s \leq \min \{\frac{2\lambda}{G^2}, \frac{1}{L}\}$, we have $\sigma_i(\bar{\xm}^{t+1})  \leq \sqrt{2\lambda s}$, therefore $\sigma_i(\xm^{t+1})  = 0$. So the zero elements of $\sigma(\xm^t)$ remain unchanged in $\sigma(\xm^{t+1})$, and  $\sp (\sigma(\xm^{t+1})) \subseteq \sp (\sigma(\xm^{t}))$, $\rk(\xm^{t+1}) \leq  \rk(\xm^{t})$. 
\end{proof}

\begin{lemma}\label{lem:2appen}
The sequence of the objective $\{F(\xm^t)\}_t$ is nonincreasing, and the following inequality holds for all $t \geq 0$:
\begin{equation}
    F(\xm^{t+1}) \leq F(\xm^t) - (\frac{1}{2s} - \frac{L}{2})\|\xm^{t+1} - \xm^t\|^2_F.
\end{equation}
\end{lemma}

\begin{proof}
    Let $\bar{\xm}^{t+1} = \xm^t - s\nabla g(\xm^t), \qm^t = -s\nabla g(\xm^t),$ thus we have $\bar{\xm}^{t+1} = \xm^t + \qm^t$,
    and 
    \begin{equation}
        \xm^{t+1} = \argmin_{\vm} \frac{1}{2s} \|\vm - \bar{\xm}^{t+1} \|^2_F + h(\vm).
    \end{equation}
    Let $\vm = \xm^t$, we get
    \begin{equation}\label{eq:ap1}
        \langle \nabla g(\xm^t), \xm^{t+1} - \xm^t \rangle + \frac{1}{2s} \|\xm^{t+1} - \xm^t\|^2_F + h(\xm^{t+1}) \leq h(\xm^t).
    \end{equation}
    In addition, we have
    \begin{equation}\label{eq:ap2}
        g(\xm^{t+1}) \leq g(\xm^t) + \langle \nabla g(\xm^t), \xm^{t+1} - \xm^t \rangle + \frac{L}{2} \|\xm^{t+1} - \xm^t\|^2_F.
    \end{equation}
Combine Eq.~(\ref{eq:ap1}) and Eq.~(\ref{eq:ap2}) together, we get
\begin{equation}
    F(\xm^{t+1}) \leq F(\xm^t) - (\frac{1}{2s} - \frac{L}{2}) \|\xm^{t+1} - \xm^t \|^2_F,
\end{equation}
    since $s \leq \frac{1}{L}$, we have $\frac{1}{2s} \geq \frac{L}{2}$, so the sequence $\{ F(\xm^t)\}_t$ is decreasing with lower bound 0. 
\end{proof}

\begin{lemma}~\cite{laurent2000adaptive}\label{lem:aappen}
    Let $Y_1, Y_2, \dots, Y_D$ be i.i.d. Gaussian random variables with 0 mean and unit variance, and $a_1, a_2, \dots, a_D$ be D positive numbers. Define $Z = \sum a_i(Y_i^2 - 1)$ and $\av = [a_1, a_2, \dots, a_D]^T$, then for any $t > 0$, 
    \begin{equation}
        \textrm{Probability} (Z \geq 2\|\av\|_2 \sqrt{t} + 2\|\av\|_{\infty} t) \leq e^{-t}.
    \end{equation}
\end{lemma}

\begin{lemma}~\cite{davidson2001local}\label{lem:bappen}
    Suppose $\am \in \mathbb{R}^{m \times n} ( m \geq n)$ is a random matrix whose entries are i.i.d. sampled from the standard Gaussian distribution \textit{N}$(0,\frac{1}{m})$, then 
    \begin{equation}
        1-\sqrt{\frac{n}{m}} \leq E(\sigma_n (\am)) \leq E(\sigma_1(\am)) \leq 1+ \sqrt{\frac{n}{m}}.
    \end{equation}
    And for any $t>0$,
    \begin{equation}
         \textrm{Probability}(\sigma_n(\am) \leq 1-\sqrt{\frac{n}{m}} - t) < e^{-\frac{mt^2}{2}},
    \end{equation}
    \begin{equation}
        \textrm{Probability}(\sigma_1(\am) \geq 1+\sqrt{\frac{n}{m}} + t) < e^{-\frac{mt^2}{2}}.
    \end{equation}
\end{lemma}

\begin{theorem}\label{thm:1appen}
Suppose $\dm \in \mathbb{R}^{d \times n} (n \geq d)$ is a random matrix with elements i.i.d. sampled from the standard Gaussian distribution \textit{N}(0,1), then 
\begin{equation}
    \textrm{Probability}(\frac{1}{L} \leq \frac{2\lambda}{G^2}) \geq 1-e^{-\frac{a^2}{2}} - ne^{-a},
\end{equation}
if 
\begin{equation}\label{eq:un}
    n \geq (\sqrt{d} + a + \sqrt{\frac{(d+2\sqrt{da}+2a)(x_0 + \lambda |S|))}{\lambda}})^2,
\end{equation}
where $x_0 = \|\ym - \dm \xm^0\|^2_F$, $S = \sp(\sigma(\xm^0))$, and $a$ can be chosen as $a_0 \textrm{log}n$ for $a_0 > 0$ to ensure that Eq.~(\ref{eq:un}) holds with high probability. 
\end{theorem}

\begin{proof}
    Based on Lemma~\ref{lem:bappen}, for any $a > 0$, with probability $\geq 1- e^{-\frac{a^2}{2}}$,
    \begin{equation}
        \sigma_{max}(\dm) > \sqrt{n} - \sqrt{d} - a,
    \end{equation}
    and by Lemma~\ref{lem:aappen}, for any $1 \leq i \leq n$ and $a > 0$, with probability $\geq 1- e^{-a}$, 
    \begin{equation}
        \|\dm_i\|_2 \leq \sqrt{d+2\sqrt{da} + 2a},
    \end{equation}
    where $\dm_i$  denotes $i$-th column of $\dm$. Then, it can be verified with the union bound that with probability $\geq 1-e^{-\frac{a^2}{2}} - ne^{-a}$, 
    \begin{equation}
        \frac{2D^2(x_0+\lambda|S|)}{\lambda} \leq 2\sigma_{max}^2(\dm),
    \end{equation}
    where $D = \max_i \|i_{th} \textrm{ column of }\dm\|_2$,
    if 
    \begin{equation}\label{eq:unres}
    n \geq (\sqrt{d} + a + \sqrt{\frac{(d+2\sqrt{da}+2a)(x_0 + \lambda |S|))}{\lambda}})^2.
\end{equation}
\end{proof}

\begin{lemma}\label{lem:3appen}
(a) All the elements of each subsequence $\xt^k$ ($k = 1, \dots, K)$ in the subsequences with shrinking support have the same support. In addition, for any $1 \leq k_1 < k_2 \leq K$ and any $\xm^{t_1} \in \xt^{k_1}$, and $\xm^{t_2} \in \xt^{k_2}$, we have $t_1 < t_2$ and $\sp(\sigma(\xm^{t_2})) \subset \sp(\sigma(\xm^{t_1}))$. 

(b) All the subsequences except for the last one, $\xt^k$ ($k = 1, \dots, K-1)$ have finite size, and $\xt^K$ have an infinite number of elements, and there exists some $t_0 \geq 0$ such that $\{\xm^t\}_{t=t_0}^\infty \subseteq \xt^K$.
\end{lemma}

\begin{proof}
    (a)
    For any $1 \leq k < K$, let $\xm^{t_1}, \xm^{t_2} \in \xt^k$ and $t_1 \neq t_2$. 
    If $t_1 < t_2$, then $\sp(\sigma(\xm^{t_2})) \subseteq \sp(\sigma(\xm^{t_1}))$ according to the support shrinkage property in Lemma~\ref{lem:1appen}. If $\sp(\sigma(\xm^{t_2})) \subset \sp(\sigma(\xm^{t_1}))$ then $|\sp(\sigma(\xm^{t_2}))| < |\sp(\sigma(\xm^{t_1}))|$, which contradicts with the definition of $\xt^k$ whose elements have the same support size. A similar argument holds for $t_2 < t_1$. Therefore, all the elements of each subsequence $\xt^k (1 \leq k \leq K)$ have the same support.

    For any $1 \leq k_1 \leq k_2 \leq K$ and any $\xm^{t_1} \in \xt^{k_1}$ and $\xm^{t_2} \in \xt^{k_2}$, note that $t_1 \neq t_2$ and $\sp(\sigma(\xm^{t_1})) \neq \sp(\sigma(\xm^{t_2}))$ since $\xt^{k_1}$ and $\xt^{k_2}$ have different support size. Suppose $t_1 > t_2$, we have $\sp(\sigma(\xm^{t_1})) \subset \sp(\sigma(\xm^{t_2}))$ and it follows that $|\sp(\sigma(\xm^{t_1}))| < 
|\sp(\sigma(\xm^{t_2}))|$, again it contradicts with the Definition~\ref{def:1}. Thus, we must have $t_1 < t_2$, and it follows that   $\sp(\sigma(\xm^{t_2})) \subset \sp(\sigma(\xm^{t_1}))$.

(b)
Suppose $\xt^k$ is an infinite sequence for some $1 \leq k \leq K-1$. We can get an infinite sequence from $\xt^k$ as follows:

We have some $\xm^{t_0} \in \xt^k$ for some $t_0 > 0$ since $\xt^k$ is not empty. Suppose we get $\{\xm^{t'_j}\}_{j' = 0}^j$ in the first $j \geq 0$ steps with increasing indices $\{t'_j\}$. Since $\xt^k$ is an infinite sequence, $\xt^k \backslash \{\xm^{t'_j}\}_{j' = 0}^j $ is still an infinite sequence. At the $(j+1)$-th step, we can find $\xm^{t_{j+1}} \in \xt^k \backslash \{\xm^{t'_j}\}_{j' = 0}^j $ with $t_{j+1} > t_j$. Therefore, we are able to get an infinite sequence $\{\xm^{t_j}\}_{j' = 0}^{\infty} \subseteq \xt^k$ with increasing indices $\{t_j\}$. With the fact that the indices $\{t_j\}$ is increasing, we can see that $\lim_{j \rightarrow \infty} t_j = \infty$.

For any element $\xm^q \in \xt^{k+1}$, there must exist some $j > 0$ such that $q \leq t_j$, according to the support shrinkage property we must have $\sp(\sigma(\xm^{t_j})) \subseteq \sp(\sigma(\xm^{q}))$, and $|\sp(\sigma(\xm^{t_j}))| \leq |\sp(\sigma(\xm^{q}))|$. 
On the other hand, since $\xm^{t_j} \in \xt^k$, we have $|\sp(\sigma(\xm^{q}))| < |\sp(\sigma(\xm^{t_j}))|$.
This contradiction shows that each $\xt^k (1\leq k \leq K-1)$ must have a finite size. 
Also, $\{\xm^t\}_{t=0}^{\infty}$ is an infinite sequence and $\{\xt^k\} _{k=1}^{K}$ form a disjoint cover of it, thus $\xt^K$ must contain infinite number of elements.  

According to the proof of (a), there exists an infinite sequence $\{\xm^{t_j}\}_{j=0}^{\infty} \subseteq \xt^K$, and $\lim_{j \rightarrow \infty} t_j = \infty$. For any $t > t_0$, there must be some $t'_j$ with $j' \geq 1$ such that $t_{j'-1} \leq t \leq t_{j'}$. Then we have 
\begin{equation}
    \sp(\sigma(\xm^{t_{j'}})) = S^* \subseteq \sp(\sigma(\xm^t)) \subseteq \sp(\sigma(\xm^{t_{j'-1}})) = S^*,
\end{equation}
therefore we have $|\sp(\sigma(\xm^t))| = |S^*|$ and $\xm^t \in \xt^K$ for any $t \geq t_0$, which is $\{\xm^{t}\}_{t=t_0}^{\infty} \subseteq \xt^K$.
\end{proof}

\begin{theorem}\label{thm:2appen}
    Suppose $s \leq \min \{\frac{2\lambda}{G^2}, \frac{1}{L}\}$. and $\xm^*$ is a limit point of $\{\xm^{t}\}_{t=0}^{\infty}$, 
    and $\sigma(\xm^*)$ is a limit point of $\{\sigma(\xm^{t})\}_{t=0}^{\infty}$, 
    then the sequence $\{\xm^{t}\}_{t=0}^{\infty}$ generated by Algorithm~\ref{alg:1} converges to $\xm^*$,
    $\xm^*$ is a critical point of F($\cdot$), and $\sp(\sigma(\xm^*)) = S^*$, where $S^*$ is the support of any element in $\xt^K$.
    Moreover, there exists $t_0 \geq 0$ such that for all $m \geq t_0$, we have
    \begin{equation}
        F(\xm^{m+1}) - F(\xm^*) \leq \frac{1}{2s(m-t_0+1)}\|\xm^{t_0} - \xm^*\|_F^2.
    \end{equation}
\end{theorem}

\begin{proof}
    Let $S^*$ denote the support of any element in $\xt^K$. First we have $\sp(\sigma(\xm^*)) \subseteq S^*$, otherwise, pick an arbitrary $i \in \sp(\sigma(\xm^*)) \backslash S^*$, then $\|\sigma(\xm^{t_j}) - \sigma(\xm^*)\|_2 \geq  |\sigma_i (\xm^*)|$ contradicts with the fact that $\sigma(\xm^{t_j}) \rightarrow \sigma(\xm^*)$.

    Moreover, suppose $\sp(\sigma(\xm^*)) \subset S^*$, we can pick an arbitrary $i \in S^* \backslash \sp(\sigma(\xm^*))$. And it can be shown that $\sigma_i(\xm^{t_j}) \rightarrow 0$. Otherwise there exists $\epsilon > 0$, for any $j$, there exists $j' \geq j$ such that $|\sigma_i(\xm^{t_{j'}})| \geq \epsilon$. It follows that $\|\sigma(\xm^{t_{j'}}) - \sigma(\xm^*) \|_2 \geq |\sigma_i(\xm^{t_{j'}}| \geq \epsilon$, contradicting with the fact that $\sigma(\xm^{t_j}) \rightarrow \sigma(\xm^*)$.

Let $\epsilon > 0$ be a sufficiently small positive number such that $sG + \epsilon < \sqrt{2\lambda s}$. Since $\sigma_i (\xm^{t_j}) \rightarrow 0$, there exists sufficiently large $j$ such that $|\sigma_i(\xm^{t_j})| < \epsilon$. Let $\bar{\xm}^{t_j+1} = \xm^{t_j} - s\nabla g(\xm^{t_j})$, then
\begin{equation}
    |\sigma_i (\bar{\xm}^{t_j+1})| \leq |\sigma_i (\xm^{t_j})| + sG < \epsilon + sG \leq \sqrt{2\lambda s}.
\end{equation}
Then according to the update rule we have $\sigma_i(\xm^{t_j+1}) = 0$, so $\sp(\sigma(\xm^{t_j+1})) \subseteq \sp(\sigma(\xm^{t_j})) \backslash \{i\}$. On the other hand, $\xm^{t_j+1} \in \xt^k$, so we have $\sp(\sigma(\xm^{t_j+1})) = \sp(\sigma(\xm^{t_j}))$. Such contradict shows that $\sp(\sigma(\xm^*)) \subset S^*$ cannot be true. So $\sp(\sigma(\xm^*)) = S^*$.

Now we will show that $\{\xm^{t}\}_{t=t_0}^{\infty}$ converges to $\xm^*$. 

For any $\vm, \um$, we have
\begin{equation}
    g(\vm) \leq g(\um) + \langle \nabla g(\um), \vm - \um \rangle + \frac{L}{2}\|\vm-\um\|_F^2,
\end{equation}
also since $g(\cdot)$ is convex, for any $\vm$ and $t \geq 0$:
\begin{equation}
    g(\xm^{t+1}) + \langle \nabla g(\xm^{t+1}), \vm - \xm^{t+1} \rangle \leq g(\vm).
\end{equation}
    Also, since
    \begin{equation}
        \xm^{t+1} = \argmin_{\vm} \frac{1}{2s} \|\vm - (\xm^t - s\nabla g(\xm^t))\|^2_F + h(\vm),
    \end{equation}
    we have
    \begin{equation}
        -\nabla g(\xm^t) - \frac{1}{s} (\xm^{t+1} -\xm^t) \in \partial h(\xm^{t+1}),
    \end{equation}
    and 
    \begin{equation}\begin{split}
        &\frac{1}{s}(\xm^{t+1} - (\xm^t - s\nabla g(\xm^t))) + \partial h(\xm^{t+1}) = 0,\\
        & \frac{1}{s}(\sum_{i \in \sp(\sigma(\xm^{t+1}))} \sigma_i \uv_i \vv_i^T - \sum_{i} \sigma_i \uv_i \vv_i^T) + \partial h(\xm^{t+1}) = 0,
    \end{split}\end{equation}
    where $\xm^t - s\nabla g(\xm^t) = \sum_i \sigma_i \uv_i \vv_i^T$ is the singular value decomposition,
    then it follows
    \begin{equation}
        \partial h(\xm^{t+1}) = \frac{1}{s} \sum_{i \notin \sp(\sigma(\xm^{t+1}))} \sigma_i \uv_i \vv_i^T,
    \end{equation}
therefore 
\begin{equation}
    \langle \partial h(\xm^{t+1}), \xm^{t+1} \rangle = 0.
\end{equation}
    For any matrix $\vm$ such that $\sp(\sigma(\vm)) = \sp(\sigma(\xm^{t+1}))$, we have $h(\vm) = h(\xm^{t+1}) + \langle \partial h(\xm^{t+1}), \xm^{t+1} \rangle$. 

For $t \geq t_0$, we have 
\begin{equation}\begin{split}\label{eq:33}
    F(\xm^{t+1}) & \leq g(\xm^t) + \langle \nabla g(\xm^t), \xm^{t+1} -\xm^t \rangle + \frac{L}{2} \|\xm^{t+1}-\xm^t\|^2_F + h(\xm^{t+1}) \\
    & \leq g(\vm) + \langle \nabla g(\xm^t), \xm^{t+1}-\vm \rangle + \langle \nabla g(\xm^t), \xm^{t+1} -\xm^t \rangle + \frac{L}{2} \|\xm^{t+1}-\xm^t\|^2_F + h(\xm^{t+1}) \\
    & = g(\vm) + \langle \nabla g(\xm^t), \xm^{t+1} - \vm \rangle + \frac{L}{2} \|\xm^{t+1}-\xm^t\|^2_F + h(\xm^{t+1}) \\
& = g(\vm) + \langle \nabla g(\xm^t), \xm^{t+1} - \vm \rangle + \frac{L}{2} \|\xm^{t+1}-\xm^t\|^2_F + h(\vm) + \langle \nabla g(\xm^t) + \frac{1}{s} (\xm^{t+1} - \xm^t), \vm-\xm^{t+1} \rangle \\
& = F(\vm) + \frac{1}{s} \langle \xm^{t+1} -\xm^t, \vm-\xm^{t+1} \rangle + \frac{L}{2} \|\xm^{t+1} - \xm^t\|^2_F \\    
& = F(\vm) + \frac{1}{s} \langle \xm^{t+1} -\xm^t, \vm-\xm^{t} \rangle -\frac{1}{s} \|\xm^{t+1} -\xm^t \|^2_F + \frac{L}{2} \|\xm^{t+1} - \xm^t\|^2_F \\  
& = F(\vm) + \frac{1}{s}  \langle \xm^{t+1} -\xm^t, \vm-\xm^{t} \rangle -(\frac{1}{s} - \frac{L}{2}) \|\xm^{t+1} - \xm^t\|^2_F  \\
    & \leq F(\vm) + \frac{1}{s} \langle \xm^{t+1} - \xm^t, \vm - \xm^t \rangle - \frac{1}{2s} \|\xm^{t+1} - \xm^t\|^2_F.
\end{split}\end{equation}
Now suppose $\sp(\sigma(\xm^*)) = \sp(\sigma(\xm^{t+1})) = S^*$, let $\vm = \xm^*$, we have
\begin{equation}
    F(\xm^{t+1}) - F(\xm^*) \leq \frac{1}{s} \langle \xm^{t+1} - \xm^t, \xm^* - \xm^t \rangle  - \frac{1}{2s} \|\xm^{t+1} - \xm^t\|^2_F = \frac{1}{2s} (\|\xm^{t} - \xm^*\|^2_F - \|\xm^{t+1} - \xm^*\|^2_F).
\end{equation}
    Now, sum the above equation over $t = t_0,\dots,m$ with $m \geq t_0$, we get
\begin{equation}
    \sum_{t = t_0}^m F(\xm^{t+1}) - F(\xm^*) \leq
    \sum_{t = t_0}^m \frac{1}{2s} (\|\xm^{t} - \xm^*\|^2_F - \|\xm^{t+1} - \xm^*\|^2_F) 
    = \frac{1}{2s} (\|\xm^{t_0} - \xm^*\|^2_F - \|\xm^{m+1} - \xm^*\|^2_F).
\end{equation}
Since $\{F(\xm^t)\}_t$ is non-increasing,
$\sum_{t = t_0}^m F(\xm^{t+1}) - F(\xm^*) > (m - t_0 +1) F(\xm^{m+1}) - F(\xm^*) $, therefore,
\begin{equation}
    F(\xm^{m+1}) - F(\xm^*) \leq \frac{1}{2s(m - t_0 +1)} (\|\xm^{t_0} - \xm^*\|^2_F - \|\xm^{m+1} - \xm^*\|^2_F) \leq 
     \frac{1}{2s(m - t_0 +1)} (\|\xm^{t_0} - \xm^*\|^2_F). 
\end{equation}

Again, since $\xm^{t+1} = \argmin_{\vm} \langle \nabla g(\xm^t), \vm - \xm^t \rangle + \frac{1}{2s}\|\vm - \xm^t\|^2_F + h(\vm)$, then 
\begin{equation}
    \langle \nabla g(\xm^t), \xm^{t+1}-\xm^t \rangle + \frac{1}{2s}\|\xm^{t+1} - \xm^t \|^2_F + h(\xm^{t+1}) 
    \leq  \langle \nabla g(\xm^t), \xm^{t}-\xm^t \rangle + \frac{1}{2s}\|\xm^{t} - \xm^t \|^2_F + h(\xm^{t}) = h(\xm^t).
\end{equation}
Therefore, 
\begin{equation}\begin{split}
    &F(\xm^{t+1}) \leq g(\vm) + \langle \nabla g(\xm^t), \xm^{t+1}-\vm \rangle  + \frac{L}{2}\|\xm^{t+1} -\xm^t\|^2_F + h(\xm^{t+1}) \\
    & \leq g(\vm)  + \langle \nabla g(\xm^t), \xm^{t+1}-\vm \rangle + \frac{L}{2}\|\xm^{t+1} -\xm^t\|^2_F + h(\xm^t) - \langle \nabla g(\xm^t), \xm^{t+1}-\xm^t \rangle - \frac{1}{2s}\|\xm^{t+1} - \xm^t \|^2_F.
\end{split}\end{equation}
Let $\vm = \xm^t$, we get $F(\xm^{t+1}) \leq F(\xm^t) - (\frac{1}{2s} - \frac{L}{2}) \|\xm^{t+1} - \xm^t\|^2_F.$
Thus, we have
\begin{equation}
    (\frac{1}{2s} - \frac{L}{2}) \sum_{t=0}^{\infty} \|\xm^{t+1} - \xm^t\|^2_F
    \leq F(\xm^0) - F(\xm^*) < \infty,
\end{equation}
then
\begin{equation}
    \|\xm^{t+1} - \xm^t\|^2_F \rightarrow 0, \textrm{ as } t \rightarrow \infty.
\end{equation}

Now we show $\xm^*$ is a critical point of $F(\cdot)$. For $t_j \geq 1$, we have 
\begin{equation}
    \nabla g(\xm^{t_j}) - \nabla g(\xm^{t_j-1}) - \frac{1}{s}(\xm^{t_j} - \xm^{t_j-1}) \in \partial F(\xm^{t_j}),
\end{equation}
when $j \rightarrow \infty$ we have 
\begin{equation}\begin{split}
   \|\partial F(\xm^{t_j})\|_F & =  \|\nabla g(\xm^{t_j}) - \nabla g(\xm^{t_j-1}) - \frac{1}{s}(\xm^{t_j} - \xm^{t_j-1}) \|_F \\
   & \leq L\| \xm^{t_j} - \xm^{t_j-1}\|_F +\frac{1}{s}\|\xm^{t_j} - \xm^{t_j-1}\|_F \rightarrow 0.
\end{split}\end{equation}
Also, when $j \rightarrow \infty$,
\begin{equation}
    F(\xm^{t_j}) = g(\xm^{t_j}) + h(\xm^{t_j}) = g(\xm^{t_j}) + \lambda|S^*| 
    \rightarrow g(\xm^*) + \lambda|S^*|  = F(\xm^*).
\end{equation}
Therefore, $0 \in \partial F(\xm^*) $ and $\xm^*$ is a critical point. 
    \end{proof}

\section{Proofs for Subsection \ref{subsec:acc}}

\begin{lemma}\label{lem:4appen}
    The sequence $\{\xm^t\}_{t=1}^{\infty}$ generated by Algorithm~\ref{alg:2} satisfies
    \begin{equation}
        \sp(\sigma(\xm^{t+1})) \subseteq \sp(\sigma(\xm^{t})), t \geq 1.
    \end{equation}
\end{lemma}
\begin{proof}
    We will prove the above lemma using mathematical induction. 

    When $t=1$, we have $\um^1 = \xm^1, \vm^1 = \xm^1$, $\xm^2 = T_{\sqrt{2\lambda s}} (\xm^1 - s\nabla g(\xm^1))$,
    using the argument in the proof of Lemma~\ref{lem:1appen}, we have 
    \begin{equation}
        \sp(\sigma(\xm^2)) \subseteq \sp(\sigma(\xm^1)). 
    \end{equation}

    Suppose  $\sp(\sigma(\xm^{t+1})) \subseteq \sp(\sigma(\xm^{t}))$ holds for all $t \leq t'$ with $t' \geq 1$, now consider the case that $t = t'+1$. Based on the update rule for $\vm^t$, we have 
    \begin{equation}
        \sp(\sigma(\vm^{t'+1})) \subseteq \sp(\sigma(\xm^{t'+1})). 
    \end{equation}
    Let $\bar{\xm}^{t'+2} = \vm^{t'+1} - s\nabla g(\vm^{t'+1})$, then $\sigma_i(\xm^{t'+2}) = 0$ for any $i \notin \sp(\sigma(\vm^{t'+1}))$ since $\sigma_i(\bar{\xm}^{t'+2}) \leq \sqrt{2\lambda s}$ for such $i$. So the zero elements in $\sigma(\vm^{t'+1})$ remain unchanged in $\sigma(\xm^{t'+2})$, and it follows that 
    \begin{equation}
        \sp(\sigma(\xm^{t'+2})) \subseteq \sp(\sigma(\vm^{t'+1})) \subseteq \sp(\sigma(\xm^{t'+1})),
    \end{equation}
    therefore $\sp(\sigma(\xm^{t+1})) \subseteq \sp(\sigma(\xm^{t}))$ holds for $t = t'+1$. 
    Based on mathematical induction it holds for all $t \geq 1$.
\end{proof}

\begin{theorem}\label{thm:3appen}
      Suppose $s \leq \min \{\frac{2\lambda}{G^2}, \frac{1}{L}\}$. and $\xm^*$ is a limit point of $\{\xm^{t}\}_{t=0}^{\infty}$ generated by Algorithm~\ref{alg:2}, 
  then there exists $t_0 \geq 1$ such that for all $m \geq t_0$, we have
    \begin{equation}
        F(\xm^{m+1}) - F(\xm^*) \leq \frac{4}{(m+1)^2}V^{t_0},
    \end{equation}
    where $V^{t_0}$ is a value defined as
    \begin{equation}\begin{split}
        V^{t_0} =  \frac{1}{2s} \|(\alpha^{t_0-1}-1)\xm^{t_0-1} - \alpha^{t_0-1}\xm^{t_0} +\xm^*\|_F^2 + (\alpha^{t_0-1})^2(F(\xm^{t_0})-F(\xm^*)).
    \end{split}\end{equation}
\end{theorem}

\begin{proof}
    According to Lemma~\ref{lem:4appen}, there exists $t_0 \geq 0$ such that $\{\xm^t\}_{t=t_0}^\infty \subseteq \xt^K$. It follows that $\sp(\sigma(\xm^*)) = S^*$. 

    When $\sp(\sigma(\wm)) = \sp(\sigma(\xm^{t+1}))$ for $t \geq t_0$, with the similar process in Eq.~(\ref{eq:33}), we get
    \begin{equation}
        F(\xm^{t+1}) \leq F(\wm) + \frac{1}{s} \langle \xm^{t+1} - \vm^t, \wm - \vm^t \rangle - (\frac{1}{s} - \frac{L}{2}) \|\xm^{t+1} - \vm^t\|^2_F,
    \end{equation}
    Let $\wm = \xm^t$ and $\wm = \xm^*$, we get
    \begin{equation}\label{eq:47}
        F(\xm^{t+1}) \leq F(\xm^t) + \frac{1}{s} \langle \xm^{t+1} - \vm^t, \xm^t - \vm^t \rangle - (\frac{1}{s} - \frac{L}{2}) \|\xm^{t+1} - \vm^t\|^2_F,
    \end{equation}
    and
    \begin{equation}\label{eq:48}
        F(\xm^{t+1}) \leq F(\xm^*) + \frac{1}{s} \langle \xm^{t+1} - \vm^t, \xm^* - \vm^t \rangle - (\frac{1}{s} - \frac{L}{2}) \|\xm^{t+1} - \vm^t\|^2_F,
    \end{equation}
    $(\alpha^t-1) \times $Eq.~(\ref{eq:47}) + Eq.~(\ref{eq:48}), we obtain
    \begin{equation}\begin{split}\label{eq:49}
        &\alpha^t F(\xm^{t+1}) - (\alpha^t -1)F(\xm^t) - F(\xm^*) \\
        \leq & \frac{1}{s} \langle \xm^{t+1} -\vm^t, (\alpha^t-1)(\xm^t-\vm^t) + \xm^* - \vm^t \rangle - \alpha^t(\frac{1}{s} - \frac{L}{2}) \|\xm^{t+1} - \vm^t\|^2_F.
    \end{split}\end{equation}
    Multiply both sides of Eq.~(\ref{eq:49}) by $\alpha^t$, and use the fact that $(\alpha^t)^2 - \alpha^t = (\alpha^{t-1})^2$, we have
    \begin{equation}\begin{split}
        & (\alpha^t)^2(F(\xm^{t+1}) - F(\xm^*)) - (\alpha^{t-1})^2(F(\xm^{t}) - F(\xm^*))  \\
        \leq  &
        \frac{1}{2s}(\|(\alpha^t-1)\xm^t - \alpha^t \vm^t + \xm^* \|^2_F - \|(\alpha^{t}-1)\xm^t - \alpha^t \xm^{t+1} + \xm^* \|^2_F).
    \end{split}\end{equation}
    Since for any matrix $\am, \bm, \cm$, when $\sp(\sigma(\am)) \subseteq \sp(\sigma(\cm))$, we have $\|\am - P_{\sp(\sigma(\cm))} (\bm)\|_F \leq \|\am-\bm\|_F$, and $\vm^t = P_{\sp(\sigma(\xm^t))} (\um^t)$, it follows that
    \begin{equation}\begin{split}\label{eq:51}
        &(\alpha^t)^2(F(\xm^{t+1}) - F(\xm^*)) - (\alpha^{t-1})^2(F(\xm^{t}) - F(\xm^*)) \\
        \leq &
        \frac{1}{2s}(\|(\alpha^t-1)\xm^t - \alpha^t \um^t + \xm^* \|^2_F - \|(\alpha^{t}-1)\xm^t - \alpha^t \xm^{t+1} + \xm^* \|^2_F).
    \end{split}\end{equation}
    For simplicity, we define $\am^{t+1} = (\alpha^t-1)\xm^t - \alpha^t \xm^{t+1} + \xm^*, \am^{t} = (\alpha^{t-1}-1)\xm^{t-1} - \alpha^{t-1} \xm^t + \xm^*$, according to the update rule for $\um^t$, we can get $\am^t = (\alpha^t-1)\xm^t - \alpha^t \um^t + \xm^*$, then based on Eq.~(\ref{eq:51}), we obtain
    \begin{equation}\label{eq:52}
        (\alpha^t)^2(F(\xm^{t+1}) - F(\xm^*)) - (\alpha^{t-1})^2(F(\xm^{t}) - F(\xm^*)) 
        \leq 
        \frac{1}{2s}(\|\am^t\|^2_F - \| \am^{t+1}\|^2_F).
    \end{equation}
    Sum Eq.~(\ref{eq:52}) over $t = t_0, \dots, m$ for $m \geq t_0$, we have
    \begin{equation}
         (\alpha^m)^2(F(\xm^{m+1}) - F(\xm^*)) - (\alpha^{t_0-1})^2(F(\xm^{t_0}) - F(\xm^*)) 
        \leq 
        \frac{1}{2s}(\|\am^{t_0}\|^2_F - \| \am^{m+1}\|^2_F)
        \leq 
        \frac{1}{2s}\|\am^{t_0}\|^2_F,
    \end{equation}
    therefore, with $\alpha^t \geq \frac{t+1}{2}$, we get
    \begin{equation}\begin{split}
        F(\xm^{m+1}) - F(\xm^*) 
        \leq &
        \frac{1}{2s(\alpha^m)^2} \|\am^{t_0}\|^2_F + \frac{(\alpha^{t_0-1})^2}{(\alpha^m)^2} (F(\xm^{t_0}) - F(\xm^*)) \\
        \leq  &
        \frac{4}{(m+1)^2}(\frac{1}{2s} \|\am^{t_0}\|^2_F + (\alpha^{t_0-1})^2(F(\xm^{t_0}) - F(\xm^*))),
    \end{split}\end{equation}
    where $\am^{t_0} = (\alpha^{t_0-1}-1)\xm^{t_0-1} - \alpha^{t_0-1}\xm^{t_0} +\xm^*$.
\end{proof}

\section{Proofs for Subsection \ref{subsec:mon}}

\begin{lemma}\label{lem:5appen}
    The sequence $\{\zm^t\}_{t=1}^{\infty}$ and $\{\xm^t\}_{t=1}^{\infty}$ generated by Algorithm~\ref{alg:3} satisfies
     \begin{equation}
        \sp(\sigma(\zm^{t+1})) \subseteq \sp(\sigma(\zm^{t})),
        \sp(\sigma(\xm^{t+1})) \subseteq \sp(\sigma(\xm^{t})), t \geq 1.
    \end{equation}
\end{lemma}

\begin{proof}
    We will prove the above lemma using mathematical induction. 

    It can be easily verified that $\sp(\sigma(\zm^2)) \subseteq \sp(\sigma(\zm^1))$. 
    
     Suppose  $\sp(\sigma(\zm^{t+1})) \subseteq \sp(\sigma(\zm^{t}))$ holds for all $t \leq t'$ with $t' \geq 1$, now consider the case that $t = t'+1$. With the similar thought process in the proof for Lemma~\ref{lem:4appen}, based on the update rule for $\wm^t$, the zero elements in $\sigma(\vm^{t'+1})$ remain unchanged in $\sigma(\zm^{t'+2})$, thus we have $\sp(\sigma(\zm^{t'+2})) \subseteq \sp(\sigma(\vm^{t'+1})) \subseteq \sp(\sigma(\zm^{t'+1}))$.

     Therefore, $\sp(\sigma(\zm^{t+1})) \subseteq \sp(\sigma(\zm^{t}))$ holds for all $t \geq 1$. 

We already show that for all $t \geq 1$, $\sp(\sigma(\xm^{t})) = \sp(\sigma(\zm^{\bar{t}}))$ for some $\bar{t} \leq t$. And based on the update rule for $\xm$, we have $\xm^{t+1} = \zm^{t+1}$ or $\xm^{t+1} = \xm^t$. 
If $\xm^{t+1} = \zm^{t+1}$, $\sp(\sigma(\xm^{t+1})) = \sp(\sigma(\zm^{t+1})) \subseteq \sp(\sigma(\zm^{\bar{t}})) = \sp(\sigma(\xm^{t}))$ since $\bar{t} \leq t < t+1$.
If $\xm^{t+1} = \xm^t$, it's easy to see  $\sp(\sigma(\xm^{t+1})) = \sp(\sigma(\xm^{t}))$.
Therefore, $\sp(\sigma(\xm^{t+1})) \subseteq \sp(\sigma(\xm^{t}))$ holds for all $t \geq 1.$
     
\end{proof}

\begin{theorem}\label{thm:4appen}
      Suppose $s \leq \min \{\frac{2\lambda}{G^2}, \frac{1}{L}\}$. and $\xm^*$ is a limit point of $\{\xm^{t}\}_{t=0}^{\infty}$ generated by Algorithm~\ref{alg:3}, 
  then there exists $t_0 \geq 1$ such that for all $m \geq t_0$, we have
    \begin{equation}
        F(\xm^{m+1}) - F(\xm^*) \leq \frac{4}{(m+1)^2}W^{t_0},
    \end{equation}
    where $W^{t_0}$ is a value defined as
    \begin{equation}\begin{split}
        W^{t_0} =  \frac{1}{2s} \|(\alpha^{t_0-1}-1)\xm^{t_0-1} - \alpha^{t_0-1}\zm^{t_0} +\xm^*\|_F^2 
         + (\alpha^{t_0-1})^2(F(\xm^{t_0})-F(\xm^*)).
    \end{split}\end{equation}
\end{theorem}
\begin{proof}
 Based on Lemma~\ref{lem:5appen},  $\{\xm^t\}_{t=0}^\infty$ forms at most $K_1 \leq |S|+1$ subsequences with shrinking support $\{\xt^k\}_{k=1}^{K_1}$, and  $\{\zm^t\}_{t=0}^\infty$ forms at most $K_2 \leq |S|+1$ subsequences with shrinking support $\{\zt^k\}_{k=1}^{K_2}$. Based on Lemma~\ref{lem:3appen}, there exists $t_1 \geq 0$ such that $\{\xm^t\}_{t=t_1}^\infty \subseteq \xt^{K_1}$, and there exists $t_2 \geq 0$ such that $\{\zm^t\}_{t=t_2}^\infty \subseteq \zt^{K_2}$. Let all the elements of $\sigma(\xt^{K_1})$ have support $S_1$, let all the elements of $\sigma(\zt^{K_2})$ have support $S_2$, we show that $S_1 = S_2$: let $t_0 = \max \{t_1, t_2\}$, then there exists $t' \geq t_0$ such that $\xm^{t'} = \zm^{t'}$, and due to the fact that $\{\xm^t\}_{t=t_1}^\infty \subseteq \xt^{K_1}$ and $\{\zm^t\}_{t=t_2}^\infty \subseteq \zt^{K_2}$, we have $S_1 = \sp(\sigma(\xm^{t'})) = \sp(\sigma(\zm^{t'})) = S_2$. 

 Let $S_1 = S_2 = S^*$, then the singular value vectors of all the elements of  $\{\xm^t\}_{t=t_0}^\infty$ and $\{\zm^t\}_{t=t_0}^\infty$ have the same support $S^*$, and $\sp(\sigma(\xm^*)) = S^*$.

Following the same process in the proof for Theorem~\ref{thm:3appen}, we get
    \begin{equation}\label{eq:62}
        F(\zm^{t+1}) \leq F(\xm^t) + \frac{1}{s} \langle \zm^{t+1} - \vm^t, \xm^t - \vm^t \rangle - (\frac{1}{s} - \frac{L}{2}) \|\zm^{t+1} - \vm^t\|^2_F,
    \end{equation}
    and
    \begin{equation}\label{eq:63}
        F(\xm^{t+1}) \leq F(\xm^*) + \frac{1}{s} \langle \zm^{t+1} - \vm^t, \xm^* - \vm^t \rangle - (\frac{1}{s} - \frac{L}{2}) \|\zm^{t+1} - \vm^t\|^2_F,
    \end{equation}
    and  $(\alpha^t-1) \times $Eq.~(\ref{eq:62}) + Eq.~(\ref{eq:63}), multiply both sides by $\alpha^t$, and use the fact that $(\alpha^t)^2 - \alpha^t = (\alpha^{t-1})^2$, we have
    \begin{equation}\begin{split}
        &(\alpha^t)^2(F(\xm^{t+1}) - F(\xm^*)) - (\alpha^{t-1})^2(F(\xm^{t}) - F(\xm^*))  \\
        \leq  &
        \frac{1}{2s}(\|(\alpha^t-1)\xm^t - \alpha^t \vm^t + \xm^* \|^2_F - \|(\alpha^{t}-1)\xm^t - \alpha^t \zm^{t+1} + \xm^* \|^2_F).
    \end{split}\end{equation}
    Based on the update rule for $\vm^t$, we have 
    \begin{equation}\begin{split}\label{eq:67}
       & (\alpha^t)^2(F(\zm^{t+1}) - F(\xm^*)) - (\alpha^{t-1})^2(F(\xm^{t}) - F(\xm^*)) \\
        \leq  &
        \frac{1}{2s}(\|(\alpha^t-1)\xm^t - \alpha^t \um^t + \xm^* \|^2_F - \|(\alpha^{t}-1)\xm^t - \alpha^t \zm^{t+1} + \xm^* \|^2_F).
    \end{split}\end{equation}

    Define $\am^{t+1} = (\alpha^t-1)\xm^t - \alpha^t \zm^{t+1} + \xm^*, \am^{t} = (\alpha^{t-1}-1)\xm^{t-1} - \alpha^{t-1} \zm^t + \xm^*$, we can get $\am^t = (\alpha^t-1)\xm^t - \alpha^t \um^t + \xm^*$, therefore,
    \begin{equation}\label{eq:68}
        (\alpha^t)^2(F(\zm^{t+1}) - F(\xm^*)) - (\alpha^{t-1})^2(F(\xm^{t}) - F(\xm^*)) 
        \leq 
        \frac{1}{2s}(\|\am^t\|^2_F - \| \am^{t+1}\|^2_F).
    \end{equation}
    Sum Eq.~(\ref{eq:68}) over $t = t_0, \dots, m$ for $m \geq t_0$, we have
    \begin{equation}
         (\alpha^m)^2(F(\zm^{m+1}) - F(\xm^*)) - (\alpha^{t_0-1})^2(F(\xm^{t_0}) - F(\xm^*)) 
        \leq 
        \frac{1}{2s}(\|\am^{t_0}\|^2_F - \| \am^{m+1}\|^2_F)
        \leq 
        \frac{1}{2s}\|\am^{t_0}\|^2_F,
    \end{equation}
    therefore, with $\alpha^t \geq \frac{t+1}{2}$, we get
    \begin{equation}
        F(\zm^{m+1}) - F(\xm^*) 
        \leq 
        \frac{4}{(m+1)^2}(\frac{1}{2s} \|\am^{t_0}\|^2_F + (\alpha^{t_0-1})^2(F(\xm^{t_0}) - F(\xm^*))),
    \end{equation}
    where $\am^{t_0} = (\alpha^{t_0-1}-1)\xm^{t_0-1} - \alpha^{t_0-1}\zm^{t_0} +\xm^*$.
\end{proof}


\end{document}